\documentclass[12pt,a4paper,reqno]{article}
\usepackage{stix}
\usepackage{amsmath,amsfonts}
\usepackage{amsthm,upref}
\usepackage{upref}
\usepackage{enumerate}
\usepackage{verbatim}
\usepackage{tikz}
\usepackage{authblk}
\usepackage[margin=3.5cm]{geometry}
\usepackage{hyperref}

\title{A criterion for the uniform eventual positivity of operator semigroups}

\makeatletter
\hypersetup{pdfauthor={Daniel Daners, Jochen Glueck}}
\hypersetup{pdftitle=\@title}
\makeatother
\author[1]{Daniel Daners}%
\author[2]{Jochen Gl\"uck\thanks{Supported by a scholarship within the
    scope of the LGFG Baden-W\"urttemberg, Germany.}}%
\affil[1]{School of Mathematics and Statistics, University of Sydney,
  NSW 2006, Australia\authorcr%
  \nolinkurl{daniel.daners@sydney.edu.au}}%
\affil[2]{Institut f\"ur Angewandte Analysis, Universit\"at Ulm,
  D-89069 Ulm, Germany\authorcr%
  \nolinkurl{jochen.glueck@uni-ulm.de}}%

\date{\today}

\newtheorem{theorem}{Theorem}[section]

\newtheorem{proposition}[theorem]{Proposition}
\newtheorem{corollary}[theorem]{Corollary}

\theoremstyle{definition}

\theoremstyle{remark}
\newtheorem{remark}[theorem]{Remark}

\numberwithin{equation}{section}

\DeclareMathOperator{\one}{\mathbf 1}

\DeclareMathOperator{\spb}{s}
\DeclareMathOperator{\spec}{\sigma}

\DeclareMathOperator{\repart}{Re}

\DeclareMathOperator{\dist}{dist}
\DeclareMathOperator{\Res}{\mathcal R}
\newcommand{\dx}{\mathrm{d}}

\newcommand{\bbN}{\mathbb{N}}

\newcommand{\bbR}{\mathbb{R}}

\newcommand{\bbC}{\mathbb{C}}

\newcommand{\calL}{\mathcal{L}}

\newcommand{\phdot}{\mathord{\,\cdot\,}}

\let\oldthebibliography\thebibliography
\renewcommand\thebibliography[1]{
  \oldthebibliography{#1}
  \setlength{\parskip}{0pt}
  \setlength{\itemsep}{0pt plus 0.3ex}
  \small
}

\begin{document}
\maketitle

\begin{abstract}
  Consider a $C_0$-semigroup $(e^{tA})_{t \ge 0}$ on a function space or, more generally, on a Banach lattice $E$. We prove a sufficient criterion for the operators $e^{tA}$ to be positive for all sufficiently large times $t$, while the semigroup itself might not be positive. This complements recently established criteria for the individual orbits of the semigroup to become eventually positive for all positive initial values. We apply our main result to study the qualitative behaviour of the solutions to various partial differential equations.
\end{abstract}

\section{Introduction}
\label{section:introduction}
Consider the abstract evolution equation
\begin{equation}
  \tag{EQ}\label{eq:ivp}
  \begin{aligned}
    \dot u(t) & = Au(t)&&\text{for $t \ge 0$,} \\
    u(0) & = u_0&&
  \end{aligned}
\end{equation}
on a Banach space $E$, where $A\colon E \supseteq D(A) \to E$ is a linear operator and the vector $u_0 \in E$ is an initial value. If the operator $A$ is the generator of a $C_0$-semigroup, which we denote by $(e^{tA})_{t \ge 0}$, then the equation \eqref{eq:ivp} has a (mild) solution for every initial value $u_0 \in E$ and this solution is given by $t \mapsto e^{tA}u_0$. If we wish to understand the qualitative behaviour of the solutions to \eqref{eq:ivp} we thus have to study the semigroup $(e^{tA})_{t \ge 0}$. We refer the reader to one of the monographs \cite{Engel2000, Engel2006, Pazy1983} for details about $C_0$-semigroups and their relation to \eqref{eq:ivp}.

In many applications $E$ is some kind of function space or, from a more abstract point of view, a Banach lattice. It is natural to ask whether the solution of \eqref{eq:ivp} takes only positive values in case that the initial value $u_0$ is positive. This is equivalent to asking whether the $C_0$-semigroup $(e^{tA})_{t \ge 0}$ is positive, a topic which has been extensively studied (see for instance \cite{Arendt1986, Batkai2017} for an overview). 

Typical examples of positive $C_0$-semigroups are those generated by first or second order differential operators with local boundary conditions such as the Dirichlet- or Neumann-Laplacian. On the other hand, many other types of differential operators, for instance higher order elliptic operators, the Laplace operator with certain non-local boundary conditions and the Dirichlet-to-Neumann operator do not generate positive semigroups. Yet, it was recently observed that -- under appropriate assumptions -- those differential operators often exhibit a certain weaker type of positive behaviour; namely, whenever the initial value $u_0$ in \eqref{eq:ivp} is positive, then the solution becomes positive for all large enough times. Such \emph{eventual positivity} was first observed for several concrete differential equations \cite{Ferrero2008, Gazzola2008, Daners2014}. In order to understand this phenomenon, the development of a general theory of eventually positive $C_0$-semigroups was initiated by J.B.~Kennedy and the present authors in several recent works, see \cite{Daners2016, Daners2016a, Daners2017, DanersPERT} and \cite[Part~III]{GlueckDISS}. Applications of the general theory developed so far include biharmonic heat equations \cite[Section~6.4]{Daners2016}, \cite[Sections~6 and~9]{Daners2016a}, heat equations with non-local boundary conditions \cite[Section~6]{Daners2016a}, \cite[Section~11.7]{GlueckDISS}, the evolution equation associated with the Dirichlet-to-Neumann operator \cite[Section~6.2]{Daners2016}, \cite[Sections~6 and~9]{Daners2016a} and delay differential equations \cite[Section~6.5]{Daners2016}, \cite[Section~9]{Daners2016a}, \cite[Sections~8.4 and~11.6]{GlueckDISS}. The notion of eventual positivity is also related to generalised Perron--Frobenius theory as discussed in \cite{Glueck2017}. More on the history and related results on eventual positivity can be found in the introduction to \cite{Daners2016a}.

\paragraph*{Individual and uniform eventual positivity.}
For most of the applications stated above, it was proved that the semigroup under consideration has a certain property which we name \emph{individual eventual strong positivity}. By this, we mean the following: let $E$ be a Banach lattice with positive cone $E_+$ and let $u \in E_+$ be a quasi-interior point of $E_+$ (in the important case where $E$ is an $L^p$ space over a $\sigma$-finite measure space and $1 \le p < \infty$, this simply means that $u(\omega) > 0$ for almost every $\omega$ in the measure space). For $f \in E$ we write $f \gg_u 0$ if there exists a number $\varepsilon > 0$ such that $f \ge \varepsilon u$. A $C_0$-semigroup $(e^{tA})_{t \ge 0}$ on $E$ is said to be \emph{individually eventually strongly positive with respect to $u$} if, for every non-zero $f \in E_+ \setminus \{0\}$, there exists a time $t_0 \ge 0$ such that $e^{tA}f \gg_u 0$ for all subsequent times $t \ge t_0$. The word \emph{individually} in this notion stresses that $t_0$ is allowed to depend on $f$. In case that $t_0$ can be chosen to be independent of $f$ we say that the semigroup is \emph{uniformly eventually strongly positive with respect to $u$}. 

In \cite{Daners2016a} and \cite[Chapter~9]{GlueckDISS} the reader can find several criteria and characterisations of individual eventual strong positivity which can be used to study this property for the above mentioned examples. On the other hand, much less is currently known about uniform eventual strong positivity. The one sufficient criterion known so far is only applicable for self-adjoint semigroups on $L^2$-spaces; it is due to the present authors and can be found in the second author's PhD thesis \cite[Theorem~10.2.1]{GlueckDISS}. In this article we present a much more general criterion which is applicable on a large class a Banach lattices and also to non-self adjoint semigroups on $L^2$-spaces (Theorem~\ref{thm:criterion-semigroups}). We achieve this by imposing certain conditions not only on the semigroup generator, but also on its dual operator. In particular, the conditions are convenient to apply on a reflexive space if the problem and its dual have the same structure. This is often the case for boundary value problems on some $L^p$ space with $1<p<\infty$. In fact most of our examples in Section~\ref{section:applications} fall into that category.

\paragraph*{Simplest case of main result and organisation of the paper.}
In the simplest case our result can be described as follows: Assume that $E=L^p(\Omega)$ for some finite measure space $(\Omega,\mu)$ and some $p\in(1,\infty)$. Then the dual space is $E'=L^q(\Omega)$ with $q\in(1,\infty)$ such that $1/p+1/q=1$. Assume that $A$ generates a strongly continuous real semigroup on $E$. Then the dual operator $A'$ of $A$ generates a strongly continuous semigroup on $E'=L^q(\Omega)$. Further suppose that the following two assumptions are satisfied:
\begin{itemize}
\item[(a)] There exist $t_1,t_2>0$ such that $e^{t_1A}L^p(\Omega)\subseteq L^\infty(\Omega)$ and $e^{t_2A'}L^q(\Omega)\subseteq L^\infty(\Omega)$.
\item[(b)] The spectral bound $\spb(A)$ is a dominant spectral value and $A$ as well as $A'$ have an eigenfunction corresponding to $\spb(A)$ that is positive and bounded away from zero on $\Omega$.
\end{itemize}
Then, there exist $t_0>0$ and $\varepsilon>0$ such that
\begin{displaymath}
  e^{tA}f\geq \varepsilon \int_\Omega f\,\dx x
  \qquad\text{for all $t\geq t_0$ and all $f\in L^p(\Omega)$.}
\end{displaymath}
The above formulation is suitable to deal with parabolic problems with Neumann-type boundary conditions where any positive eigenfunction is bounded away from zero on $\overline\Omega$. To deal with more general problems including boundary value problems with Dirichlet boundary conditions we need a few  preliminary results on duality in Banach lattices. These are provided in Section~\ref{section:domination-and-duality}. Our main result in full generality and several convenient reformulations are proved in Section~\ref{section:eventually-positive-semigroups}. In Section~\ref{section:applications} we illustrate how our results apply to various differential operators, thereby improving a number of results from \cite[Section~6]{Daners2016a} and \cite[Chapter~11]{GlueckDISS}.

We assume that the reader is familiar with standard $C_0$-semigroup theory and with the basic concepts of real and complex Banach lattices. The most important examples of Banach lattices where our results can be applied to are $L^p$-spaces over $\sigma$-finite measure spaces, where $1 \le p < \infty$. Standard references for the theory of Banach lattices are, among others, \cite{Schaefer1974} and \cite{Meyer-Nieberg1991}. Standard references for the theory of $C_0$-semigroups are, for instance, \cite{Pazy1983}, \cite{Engel2000} and \cite{Engel2006}. Apart from the notions defined above, all further notation and terminology is introduced when needed.

\section{Domination and duality}
\label{section:domination-and-duality}
Our results depend on some regularising properties of semigroups. These regularising properties are inspired by properties of operators mapping $L^1$ to $L^p$ or $L^p$ to $L^\infty$. More precisely, let $p \in (1,\infty)$ and consider the $L^p$-space over a finite measure space. If $T$ is a bounded linear operator from $L^p$ to $L^\infty$, then it is well-known that $T$ is compact from $L^p$ to $L^p$. By duality, if $T$ is bounded from $L^1$ to $L^p$, then the restriction $T|_{L^p}: L^p \to L^p$ is compact.

Similar compactness properties can be obtained for certain operators on
Banach lattices based on two abstract constructions. These constructions
will be required to state and prove our main theorem in
Section~\ref{section:eventually-positive-semigroups}. We first motiviate
the well known constructions for $L^p$-spaces, then generalise to complex Banach lattices. If $(\Omega,\mu)$ is a
finite measure space and $p \in (1,\infty)$, then we have the continuous
and dense injections
\begin{equation}
  \label{eq:lp-inclusions}
  L^\infty(\Omega,\mu)
  \hookrightarrow L^p(\Omega,\mu)
  \hookrightarrow L^1(\Omega,\mu).
\end{equation}
These injections can be constructed directly from suitable subspaces of $L^p(\Omega,\mu)$. For the first construction let $\one \in L^p(\Omega,\mu)$ denote the constant function with value $1$. Then $L^\infty(\Omega,\mu)$ consists of exactly the functions $f \in L^p(\Omega,\mu)$ whose modulus is dominated by a multiple of $\one$, and the $L_\infty$-norm of each such $f$ is given by
\begin{displaymath}
  \|f\|_\infty = \inf \{c\geq 0\colon |f|\leq c\one\}.
\end{displaymath}
For the second construction let $\varphi$ denote the functional on $L^p(\Omega,\mu)$ given by $\langle \varphi,f\rangle := \int_\Omega f\,\dx \mu$ for all $f\in L^p(\Omega,\mu)$. Then the $L_1$-norm of each $f \in L^p(\Omega,\mu)$ is given by
\begin{displaymath}
  \|f\|_1 = \langle \varphi, |f|\rangle,
\end{displaymath}
and $L^1(\Omega,\mu)$ is the completion of $L^p(\Omega,\mu)$ with respect to this norm. Similar constructions apply to an arbitrary real or complex Banach lattice $E$ taking the role of $L^p(\Omega)$. For the first construction fix $u \in E_+$.  Letting $u$ play the role of $\one$ above we define 
\begin{displaymath}
  \begin{aligned}
    E_u &:= \bigl\{f \in E\colon
    \text{there exists $c\geq 0$ such that $|f|\leq c u$}\bigr\},\\
    \quad \|f\|_u &:= \inf \{c \ge 0\colon|f| \le cu\}
    \quad \text{for all $f \in E_u$.}
  \end{aligned}
\end{displaymath}
The space $E_u$ is called the \emph{principal ideal in $E$ generated by $u$} and $\|\phdot\|_u$ is called the \emph{gauge norm with respect to $u$} on $E_u$. It is a norm on $E_u$ which renders $E_u$ a Banach lattice and, moreover, an AM-space with unit $u$. This means that there exists a compact Hausdorff space $K$ and an isometric Banach lattice isomorphism $E_u \to C(K)$ mapping $u$ to $\one$.

For the second construction fix a stricly positive functional $\varphi \in E'$. By strictly positive we mean that $\langle \varphi, f\rangle > 0$ for all $f \in E_+ \setminus \{0\}$. If we define
\begin{displaymath}
  \|f\|_\varphi := \langle \varphi, |f| \rangle
  \quad \text{for all $f \in E$,}
\end{displaymath}
then $\|\phdot\|_\varphi$ is a norm an $E$. However, $E$ is not in general complete with this norm. We denote the completion of $E$ with 
respect to this norm by $E^\varphi$ and, for convenience, we denote the norm on $E^\varphi$ again by $\|\phdot\|_\varphi$. Then $E^\varphi$ is a Banach lattice and the norm $\|\phdot\|_\varphi$ is additive on the positive cone, by which we mean that $\|f+g\|_\varphi = \|f\|_\varphi + \|g\|_\varphi$ for all $f,g \in E^\varphi_+$. In other words, $E^\varphi$ is an \emph{AL-space} and hence, there exists an isometric Banach lattice isomorphism from $E^\varphi$ to an $L^1$-space over some not necessarily $\sigma$-finite measure space.

From now on, we always endow $E_u$ with the gauge norm $\|\phdot\|_u$ and $E^\varphi$ with the norm $\|\phdot\|_\varphi$. The canonical injections
\begin{equation}
  \label{eq:bl-inclusions}
  E_u \hookrightarrow E \hookrightarrow E^\varphi
\end{equation}
are continuous and they are injective lattice homomorphisms. The second embedding has always a dense range by construction. Whether or not the first embedding is dense depends on the choice of $u>0$. We call the points $u\in E_+$ a \emph{quasi-interior point} of $E_+$ if that embedding is dense; see \cite[Definition~1.2.12(iii)]{Meyer-Nieberg1991}. Note that~\eqref{eq:bl-inclusions} is the analogue of~\eqref{eq:lp-inclusions} in our abstract setting.

Given a bounded linear operator $T$ on $E$, there are two questions canonically related to~\eqref{eq:bl-inclusions}, namely (1) whether $T$ maps $E$ into $E_u$ and (2) whether $T$ extends to a continuous operator from $E^\varphi$ to $E$. Question~(1) plays a fairly important role in characterisation theorems for individually eventually positive semigroups (see for instance \cite[Section~5]{Daners2016a}). The following proposition shows that the second question is related to a dual version of the first one. To state the proposition we recall that the dual space $E'$ of a Banach lattice $E$ is itself a Banach  lattice. Hence, if $\varphi \in E'$ is a positive functional, that is, $\langle \varphi, f\rangle \ge 0$ for all $f \in E_+$, then the principal ideal $(E')_\varphi$ is well-defined. If $E$ and $F$ are Banach spaces, we denote the space of all bounded linear operators from $E$ to $F$ by $\calL(E,F)$, and we write $\calL(E) = \calL(E,E)$ for short.
\begin{proposition}
  \label{prop:domination-and-duality}
  Let $E$ be a real or complex Banach lattice, let $T\in\calL(E)$ and let $\varphi \in E'$ be a strictly positive functional. Assume that $T'E' \subseteq (E')_{\varphi}$.
  
  Then $T'\in\calL\bigl(E',(E')_\varphi\bigr)$. Moreover, $T$ extends to a bounded linear operator $E^{\varphi} \to E$, which we also denote by $T$ and whose norm satisfies the estimate $\|T\|_{E^{\varphi} \to E} \le \|T'\|_{E' \to (E')_\varphi}$. 
\end{proposition}
\begin{proof}
  First note that $T'\in\calL\bigl(E',(E')_\varphi\bigr)$ due to the closed graph theorem. If $\psi \in E'$ with $\|\psi\|_{E'}\leq 1$, then $|T'\psi|\leq \|T'\|_{E' \to (E')_\varphi}\varphi$. Hence, 
  \begin{displaymath}
    \|Tf\|
    = \sup_{\|\psi\| \le 1} |\langle T'\psi, f\rangle|
    \leq\sup_{\|\psi\|\leq 1}\langle |T'\psi|,|f|\rangle
    \leq\|T'\|_{E'\to(E')_\varphi}\langle \varphi, |f| \rangle
  \end{displaymath}
  for all $f \in E$, which proves the assertion.
\end{proof}
Next we show a somewhat surprising compactness property. Let $E$ be a Banach lattice, let $T \in \calL(E)$ and let $u \in E_+$. Under the assumption that $TE \subseteq E_u$ it was shown in \cite[Theorem~2.3(i)]{Daners2017} that $T^2$ is compact, provided that the Banach lattice $E$ has order continuous norm. If in addition $E$ is reflexive, then it follows that $T$ itself is compact; see \cite[Theorem~2.3(ii)]{Daners2017}. This generalises the compactness results on $L^p$ mentioned at the beginning of the section. If, on the other hand, $E$ does not have order continuous norm, such an assertion is clearly false: for example, the domination property $TE \subseteq E_u$ is always fulfilled if $E=C(K)$ for some ompact Hausdorff space $K$ and if $u=\one$. However, not every bounded linear operator on $C(K)$ has a compact power. It is our aim to show that the situation changes if we assume that $T$ and its dual both fulfil a domination condition. We will show that in this case, $T$ is automatically power compact, regardless of whether the norm on $E$ is order continuous or not. In fact, we have the following slightly more general result.

\begin{theorem}
  \label{thm:domination-and-compactness}
  Let $E$ be a real or complex Banach lattice, let $u \in E_+$ and let $\varphi \in E'$ be a strictly positive positive functional. Let $T_1,\dots,T_4\in\calL(E)$ such that $T_kE \subseteq E_u$ and $T_k'E' \subseteq (E')_\varphi$ for all $k \in \{1,2,3,4\}$. Then $T_4T_3T_2T_1\in\calL(E)$ is compact.
\end{theorem}
\begin{proof}
  We first prove the theorem in case that the underlying scalar field is real. According to Proposition~\ref{prop:domination-and-duality}, each operator $T_k$ extends to a bounded linear operator $T_k\in\calL(E^\varphi,E)$. Hence, $T_3T_2\colon E^\varphi\to E_u \subseteq (E^\varphi)_u$. Since the norm on $E^\varphi$ is additive on the positive cone, it easily follows that $E^\varphi$  has order continuous norm and thus, $T_3T_2\colon E^\varphi\to E^\varphi$ is weakly compact as proved in \cite[Lemma~2.1]{Daners2017}, that is, it maps the unit ball of $E^\varphi$ to a relatively weakly compact subset of $E^\varphi$. Using again that $E^\varphi$ is an AL-space, we conclude that $E^\varphi$ has the Dunford--Pettis property according to~\cite[Definition~3.7.6 and Proposition~3.7.9]{Meyer-Nieberg1991} and hence, $T_3T_2\colon E^\varphi\to E^\varphi$ is a Dunford--Pettis operator; see \cite[Definition~3.7.6(i)]{Meyer-Nieberg1991}. By~\cite[Proposition~3.7.11(ii)]{Meyer-Nieberg1991} this implies that $T_3T_2$ maps order intervals in $E^\varphi$ to relatively compact subsets of $E^\varphi$. 
	
  As $T_1E\subseteq E_u$ the closed graph theorem implies that $T_1\in\calL(E,E_u)$. Hence, $T_1$ maps the unit ball of $E$ into an order interval of $E$ and hence into an order interval of $E^\varphi$. As a consequence $T_3T_2T_1$ maps the unit ball of $E$ into a relatively compact subset of $E^\varphi$. Finally, using that $T_4\in\calL(E^\varphi,E)$ we conclude that $T_4T_3T_2T_1\in\calL(E)$ is compact.
	
  The case of complex scalars can easily derived from the real case by a similar argument as detailed in the last part of the proof of \cite[Lemma~9.2.1]{GlueckDISS}.
\end{proof}
\begin{corollary}
  \label{cor:domination-and-compactness}
  Let $E$ be a real or complex Banach lattice, let $u \in E_+$ and let $\varphi \in E_+'$ be strictly positive. Let $T\in\calL(E)$ such that $TE \subseteq E_u$ and $T'E' \subseteq (E')_\varphi$. Then $T^4$ is a compact operator from $E$ to $E$.
\end{corollary}

We close this section with the remark that there is a theory of so-called \emph{cone absolutely summing} and \emph{majorizing} linear operators which is related to the topics discussed above. For details we refer the interested reader for instance to \cite[Section~IV.3]{Schaefer1974} and \cite[Section~2.8]{Meyer-Nieberg1991}.

\section{Eventually positive semigroups} \label{section:eventually-positive-semigroups}
Theorem~\ref{thm:criterion-semigroups} below is our main result. To state it in a convenient form, we use the following conventions. 

Let $E$ be a complex Banach space. For a vector $u \in E$ and a functional $\varphi$ in the dual space $E'$ we use the common notation $u \otimes \varphi$ for the operator $E \to E$ given by $(u \otimes \varphi)f = \langle \varphi, f\rangle u$ for all $f \in E$. If $A\colon E \supseteq D(A) \to E$ is a linear operator, then $\spb(A) := \sup\{\repart \lambda\colon \lambda \in \spec(A)\} \in [-\infty,\infty]$ is called the \emph{spectral bound}, where $\spec(A)$ is the spectrum of $A$. The spectral bound  $\spb(A)$  is called a \emph{dominant spectral value} of $A$ if $\spb(A)\in\spec(A)$ and $\repart \lambda < \spb(A)$ for all further $\lambda\in\spec(A)$. Now, let $E$ be a complex Banach lattice and let $(e^{tA})_{t \ge 0}$ be a $C_0$-semigroup on $E$. We say that this semigroup is \emph{real} if each operator $e^{tA}$ ($t \ge 0$) leaves the real part $E_\bbR$ of $E$ invariant. This is equivalent to stipulating that the operator $A$ is \emph{real}, meaning that $D(A) = D(A) \cap E_\bbR + i D(A) \cap E_\bbR$ and that $A$ maps $D(A) \cap E_\bbR$ to $E_\bbR$.

If $(e^{tA})_{t \ge 0}$ is a $C_0$-semigroup on a Banach space $E$, then the dual semigroup $\big((e^{tA})'\big)_{t \ge 0}$ on the dual Banach space $E'$ is weak${}^*$-continuous. The dual operator $A'$ is the \emph{weak${}^*$-generator} of the dual semigroup (see \cite[Section~II.2.5]{Engel2000} for details), so we use the notation $\big((e^{tA})'\big)_{t \ge 0} =: (e^{tA'})_{t \ge 0}$. Recall that the dual semigroup is automatically a $C_0$-semigroup in case that $E$ is reflexive.

\begin{theorem}
  \label{thm:criterion-semigroups}
  Let $(e^{tA})_{t \ge 0}$ be a real $C_0$-semigroup on a complex Banach lattice $E$ whose generator $A$ has non-empty spectrum $\spec(A)$. Let $u \in E_+$ be a quasi-interior point, let $\varphi \in E'_+$ be a strictly positive functional and assume that the following assumptions are fulfilled:
  \begin{enumerate}[\upshape (a)]
  \item There exist times $t_1,t_2 \ge 0$ for which we have $e^{t_1A}E \subseteq E_u$ and $e^{t_2A'}E' \subseteq (E')_{\varphi}$.
  \item The spectral bound $\spb(A)$ is a dominant spectral value of $A$, the eigenspace $\ker (\spb(A) - A)$ is one-dimensional and contains a vector $v \gg_u 0$ and the dual eigenspace $\ker (\spb(A) - A')$ contains a functional $\psi \gg_\varphi 0$.
  \end{enumerate}
  Then there exist $t_0 \ge 0$ and $\varepsilon > 0$ such that $e^{tA} \ge \varepsilon (u \otimes \varphi)$ for all $t \ge t_0$. In particular, the semigroup $(e^{tA})_{t \ge 0}$ is uniformly eventually strongly positive with respect to $u$. Moreover, the semigroup is eventually compact, that is, $e^{tA}$ is a compact operator for all sufficiently large $t$.
\end{theorem}

It is instructive to compare Theorem~\ref{thm:criterion-semigroups} with \cite[Theorem~5.2 and Corollary~3.3]{Daners2016a} where a similar criterion is given for individual eventual strong positivity. What gives us uniform eventual positivity here are the assumptions on the dual semigroup, namely that $e^{t_2A'}E \subseteq E_\varphi$ and that the eigenvector $\psi$ of $A'$ is not only strictly positive, but fulfils $\psi \gg_\varphi 0$.

Before we prove Theorem~\ref{thm:criterion-semigroups}, a few comments on its assumptions are in order. The assumption that the semigroup be real is not particularly restrictive. It is for instance fulfilled by each semigroup which is generated by a differential operator with real coefficients. Assumption~(a) is more interesting and can be considered as a kind of abstract Sobolev embedding theorem. It is important to note that condition~(a) implies a strong relationship between the vectors $u$ and $v$, as well as between $\varphi$ and $\psi$. By assumption we have $v \gg_u 0$ and $\psi \gg_\varphi 0$. On the other hand, it follows from $e^{t_1A}E \subseteq E_u$ that $e^{t_1 \spb(A)} v = e^{t_1A}v \in E_u$. Hence, we also have the converse estimate $u \gg_v 0$. Similarly, we see that $\varphi \gg_\psi 0$. Thus, there exist real numbers $c_1,c_2,d_1,d_2 > 0$ such that
\begin{equation}
  \label{eq:equivalent-vectors}
  c_1 u \le v \le c_2 u \quad \text{and} \quad d_1 \varphi \le \psi \le d_2 \varphi.
\end{equation}
Thus, we may think of $u$ and $\varphi$ as some kind of ``deformed'' eigenvectors of $A$ and $A'$, and in the statement of the theorem we could replace all occurrences of $u$ and $\varphi$ with $v$ and $\psi$ throughout. The reason why we consider the additional vectors $u$ and $\varphi$ is that, for concrete operators $A$, it is often difficult to compute the eigenvectors $v$ and $\psi$ explicitly. On the other hand, there is often a quite canonical choice for the vectors $u$ and $\varphi$ which is essentially determined by the boundary conditions encoded in the domain $D(A)$. We refer the reader to Section~\ref{section:applications} where this is demonstrated in several examples.

Assumption~(b) in Theorem~\ref{thm:criterion-semigroups} can be checked by several methods, depending on the particular operator under consideration. On some occasions an explicit computation or estimate of the eigenvectors might be possible while in other examples it is possible to compute the resolvents $\Res(\lambda,A):=(\lambda I-A)^{-1}$ of $A$ and of $A'$ for some $\lambda>\spb(A)$ and to employ results about eventual positivity of resolvents, such as for instance \cite[Proposition~3.2]{DanersPERT} and~\cite[Theorem~4.4 and Corollary~3.3]{Daners2016a}. Again, we refer to Section~\ref{section:applications} for examples.

\begin{proof}[Proof of Theorem~\ref{thm:criterion-semigroups}]
  We have $e^{(t_1+t_2)A}E \subseteq E_u$ and $e^{(t_1+t_2)A'}E' \subseteq (E')_\varphi$ and hence $e^{4(t_1+t_2)A}$ is compact according to Corollary~\ref{cor:domination-and-compactness}. This shows that the semigroup is eventually compact.
	 
  For the rest of the proof we assume without loss of generality that $\spb(A) = 0$. Since the semigroup is eventually compact, the spectral value $0$ is a pole of the resolvent $\Res(\phdot,A)$ \cite[Corollary~V.3.2]{Engel2000} and according to \cite[Corollary~3.3]{Daners2016a} it follows from assumption~(b) that the corresponding pole order equals $1$. Let $P$ be the spectral projection of $A$ associated with $0$. After rescaling $v$ such that $\langle \psi, v\rangle = 1$ we have $P = \psi \otimes v$.
	
  Now we prove that $e^{(t+t_1+t_2)A}(I-P)$ extends to a continuous linear operator from $E^\psi$ to $E_v$ whose norm converges to $0$ as $t \to \infty$. To this end, first note that $e^{tA}(1-P)$ converges to $0$ with respect to the operator norm as $t \to \infty$; this follows from the assumption that $\spb(A) = 0$ be a dominant spectral value and from the fact that the semigroup is eventually compact \cite[Corollary~V.3.2]{Engel2000}. Moreover, using \eqref{eq:equivalent-vectors} we observe that $E_u = E_v$ and $E^\psi = E^\varphi$, which means that the vector spaces coincide and that the norms are equivalent. Hence, $e^{t_1A}$ is an operator from $E$ to $E_v$ which is continuous due to the closed graph theorem, and $e^{t_2A}$ extends to a continuous operator from $E^\psi$ to $E$ according to Proposition~\ref{prop:domination-and-duality}. This proves that $e^{(t+t_1+t_2)A}(1-P)$ extends to a continuous operator from $E^\psi$ to $E_v$ for every $t \ge 0$. We therefore have
  \begin{displaymath}
    \|e^{(t+t_1+t_2)A}(1-P)\|_{E^\psi \to E_v}
    \leq \|e^{t_1A}\|_{E \to E_v} \|e^{tA}(1-P)\|_{E \to E}
    \|e^{t_2A}\|_{E^\psi \to E} \to 0
  \end{displaymath}
  as $t \to \infty$. Fix $\delta \in (0,1)$. For all sufficiently large $t$, $t \ge \tau$ say, we thus obtain $\|e^{(t+t_0+t_1)A}(1-P)\|_{E^\psi \to E_v} \le \delta$. Hence
  \begin{displaymath}
    \|e^{tA}(1-P)f\|_v \le  \delta\|f\|_\psi
    = \delta \langle \psi, f \rangle 
  \end{displaymath}
  for all $t \ge t_0 + t_1 + \tau$ and all $f \in E_+$. Since our semigroup is real, this implies that
  \begin{displaymath}
    |e^{tA}(1-P)f| \le \delta \langle \psi, f\rangle v,
    \quad\text{hence}\quad
    e^{tA}(1-P)f \ge - \delta \langle \psi, f \rangle v.
  \end{displaymath}
  for all such $t$ and $f$. Note that $e^{tA}f = \langle \psi, f\rangle v + e^{tA}(1-P)f$ for all times $t \ge 0$ and all $f \in E$ and hence we conclude that $e^{tA}f \ge (1-\delta)\langle \psi, f \rangle v = (1-\delta) (v \otimes \psi)f$ for all $f \in E_+$ and all $t \ge t_0 + t_1 + \tau$. This proves the assertion with $\varepsilon = c_1d_1 (1-\delta)$, where $c_1$ and $d_1$ are taken from~\eqref{eq:equivalent-vectors}.
\end{proof}

\begin{remark}
	The proof of Theorem~\ref{thm:criterion-semigroups} actually shows that, for each fixed $\delta \in (0,1)$, we have $e^{tA} \ge (1-\delta) (v \otimes \psi)$ for all sufficiently large times $t$.
\end{remark}

Let us now briefly discuss how to check the conditions $e^{t_1A}E \subseteq E_u$ and $e^{t_2A'}E' \subseteq (E')_{\varphi}$ in Theorem~\ref{thm:criterion-semigroups}(a). For further information concerning this kind of conditions we also refer to \cite{Daners2017}; compare also \cite[Remark~9.3.4]{GlueckDISS}

\begin{remark}
  \label{rem:domination-condition}
  Let $(e^{tA})_{t \ge 0}$ be a $C_0$-semigroup on a complex Banach space $E$, let $u \in E_+$ be a quasi-interior point and let $\varphi \in E'_+$ be a strictly positive functional.
  \begin{enumerate}[(a)]
  \item Assume that the semigroup is analytic and consider the subspace
    \begin{displaymath}
      D(A^\infty) := \bigcap_{n \in \bbN} D(A^n)
    \end{displaymath}
    of $E$. If $D(A^\infty) \subseteq E_u$, then $e^{tA}E \subseteq D(A^\infty) \subseteq E_u$ for all $t > 0$.
		
    Moreover, note that the dual semigroup $(e^{tA'})_{t \ge 0}$ is also analytic (although it might not be a $C_0$-semigroup). If we use that $A'$ is the weak$^*$-generator of the dual semigroup, then it is easy to see that also $e^{tA'}E' \subseteq D\big((A')^\infty\big)$ for all $t > 0$. Hence, if $D\big((A')^\infty\big) \subseteq E'_\varphi$, then $e^{tA'}E' \subseteq E'_\varphi$ for all $t > 0$.
		
  \item Now only assume that $(e^{tA})_{t \ge 0}$ is eventually differentiable, that is, there exists a time $t_0 \ge 0$ such that $e^{tA}E \subseteq D(A)$ for all $t > t_0$. If there exists an integer $n \in \bbN$ for which we have $D(A^n)\subseteq E_u$, then it follows that $e^{tA}E \subseteq D(A^n) \subseteq E_u$ for all $t > nt_0$.
		
    We also note that the mapping $t \to e^{tA}$ is differentiable on $(t_0,\infty)$ with respect to the operator norm, see \cite[Exercise~II.4.16(1)]{Engel2000}. Hence, the same is true for the mapping $t \mapsto e^{tA'}$ and thus, $e^{tA'}E' \subseteq D(A')$ for all $t > t_0$ as $A'$ is the weak${}^*$-generator of $(e^{tA'})_{t \ge 0}$. This implies that $e^{tA'}E' \subseteq D\big((A')^n\big)$ for all $n \in \bbN$ and all $t > nt_0$. Therefore, we have $e^{tA'}E' \subseteq E'_\varphi$ for all sufficiently large $t$ in case that $D\big((A')^n\big)$ is contained in $E'_\varphi$ for at least one integer $n \in \bbN$.
		
  \item Let $E = C(K)$ be the space of continuous functions on a compact Hausdorff space $K$. The fact that $u$ is a quasi-interior point of $E_+$ implies that, actually, $u \ge \varepsilon \one_K$ for some $\varepsilon > 0$. Hence, $E_u = E$, so the condition $e^{tA}E \subseteq E_u$ is always fulfilled for each $t \ge 0$ in this case.
		
  \item Somewhat dually to (c), let $E = L^1(\Omega,\mu;\bbC)$ for a $\sigma$-finite measure space $(\Omega,\mu)$. Let $\one \in L^\infty(\Omega,\mu;\bbC) = E'$ be the constant function with value $1$. In contrast to (c), the assumption that $\varphi$ be strictly positive does not necessarily imply that $\varphi \ge \varepsilon \one$ for some $\varepsilon > 0$ (in fact, $\varphi$ is a strictly positive functional if and only if $\varphi(\omega) > 0$ for almost every $\omega \in \Omega$). If, however, there exists a number $\varepsilon > 0$ such that $\varphi \ge \varepsilon \one$, then we have $E'_\varphi = E'$, so the condition $e^{tA'}E' \subseteq E'_\varphi$ automatically holds for every $t \ge 0$.
  \end{enumerate}
\end{remark}

Let us give a few consequences of Theorem~\ref{thm:criterion-semigroups} in the following corollaries. The first corollary deals with the relation between individual and uniform eventual positivity.

\begin{corollary}
  \label{cor:ind-vs-unif}
  Let $(e^{tA})_{t \ge 0}$ be a real $C_0$-semigroup on a complex Banach lattice $E$ whose generator $A$ has non-empty spectrum $\spec(A)$. Let $u \in E_+$ be a quasi-interior point, let $\varphi \in E'_+$ be a strictly positive functional and assume that there exist times $t_1,t_2 \ge 0$ for which we have $e^{t_1A}E \subseteq E_u$ and $e^{t_2A'}E' \subseteq (E')_{\varphi}$. 
	
  If $(e^{tA})_{t \ge 0}$ is individually eventually strongly positive with respect to $u$ and if there exists a functional $\psi \gg_\varphi 0$ in $\ker(\spb(A) - A')$, then $(e^{tA})_{t \ge 0}$ is uniformly eventually strongly positive with respect to $u$.
\end{corollary}
\begin{proof}
  We only have to show that assumption~(b) of Theorem~\ref{thm:criterion-semigroups} is fulfilled. As before it follows from Corollary~\ref{cor:domination-and-compactness} that the operator $e^{4(t_1+t_2)A}$ is compact, so the semigroup $(e^{tA})_{t \ge 0}$ is eventually compact. Thus, each spectral value of $A$ is a pole of its resolvent, and every vertical strip in the complex plane which has finite width contains at most finitely many spectral values of $A$ \cite[Corollary~V.3.2]{Engel2000}. Hence, the assumptions of \cite[Theorem~5.2]{Daners2016a} are fulfilled. By part~(ii) of this theorem we conclude that $\spb(A)$ is a dominant spectral value of $A$ and that the spectral projection $P$ associated with $\spb(A)$ fulfils $Pf \gg_u 0$ for each $f \in E_+ \setminus \{0\}$. Due to this property of the spectral projection we can apply \cite[Corollary~3.3]{Daners2016a} which yields that the eigenspace $\ker(\spb(A)-A)$ is one-dimensional and that it contains a vector $v \gg_u 0$. Hence, assumption~(b) of Theorem~\ref{thm:criterion-semigroups} is fulfilled and the assertions follows.
\end{proof}
The second corollary deals with self-adjoint semigroups on $L^2$-spaces. Note that, until now, we have always dealt with the \emph{Banach space dual} of the generator $A$; however, the term \emph{self-adjointness} refers to the \emph{Hilbert space adjoint} of operators. Thus, for a proper understanding of the corollary and its proof, a brief discussion of the relation between those two notions is in order.

Let $(\Omega,\mu)$ be a $\sigma$-finite measure space and set $H := L^2(\Omega,\mu;\bbC)$; we consider a densely defined linear operator $A: H \supseteq D(A) \to H$ with Banach space dual $A'\colon H' \supseteq D(A') \to H'$. If we identify $H'$ with $H$ by means of the Riesz--Fr{\'e}chet representation theorem, the operator $A'$ induces an operator $A^*\colon H \supseteq D(A^*) \to H$ which is called the Hilbert space adjoint of $A$. On the other hand, there is also a ``Banach space way'' (instead of a ``Hilbert space way'') to identify $H'$ with $H$, namely we may identify each vector $z \in H$ with the functional $H \ni x \mapsto \int_\Omega z x\,\dx\mu \in \bbC$; this identification is a linear (instead of an anti-linear) isomorphism between $H$ and $H'$ and it is compatible with the usual identification of $(L^p(\Omega,\mu;\bbC)'$ and $L^p(\Omega,\mu;\bbC)$ (where $1/p + 1/q = 1$). By means of this identification, $A'$ also induces an operator on $H$ which we again denote, by abuse of notation, as $A'$, and which we also refer to as the Banach space dual of $A$. It is not difficult to see that
\begin{displaymath}
  D(A^*) = \{u \in H\colon\overline{u} \in D(A')\},
  \quad
  A^* u = \overline{A' \, \overline{u}}
  \quad \text{for all $u \in D(A^*)$.}
\end{displaymath}
In particular, we have $A^* = A'$ in case that $A$ (and thus $A'$) is real.
\begin{corollary}
  \label{cor:self-adjoint-case}
  Let $(\Omega,\mu)$ be a $\sigma$-finite measure space and let $(e^{tA})_{t \ge 0}$ be a real $C_0$-semigroup on $H := L^2(\Omega,\mu;\bbC)$ with self-adjoint generator $A$. Let $u \in H_+$ be a quasi-interior point which means that $u(\omega) > 0$ for almost all $\omega \in \Omega$. Assume that $e^{t_1A}\subseteq H_u$ for some $t_1>0$ or that $D(A^\infty) \subseteq H_u$. Then the following assertions are equivalent:
  \begin{enumerate}[\upshape (i)]
  \item The semigroup $(e^{tA})_{t \ge 0}$ is uniformly eventually strongly positive with respect to $u$.
  \item The semigroup $(e^{tA})_{t \ge 0}$ is individually eventually strongly positive with respect to $u$.
  \item The space $\ker(\spb(A) - A)$ is one-dimensional and contains a vector $v \gg_u 0$.
  \end{enumerate}
\end{corollary}

Note that assertion~(iii) in the above corollary makes sense since the spectrum of a self-adjoint operator is always non-empty, so $\spb(A) \in \bbR$. Our proof of Corollary~\ref{cor:self-adjoint-case} below is a simple consequence of Theorem~\ref{thm:criterion-semigroups}. However, the corollary can also be shown by different methods, using the theory of Hilbert--Schmidt operators \cite[Theorem~10.2.1]{GlueckDISS}. We note in passing that Corollary~\ref{cor:self-adjoint-case} can be formulated in a slightly more abstract setting, using the theory of Hilbert lattices, which is equivalent to considering $L^2$-space over arbitrary measure spaces. However, the quasi-interior points in such spaces are not easy to characterise, so we stick to the case of $\sigma$-finite measure spaces here which is most important in applications.

\begin{proof}[Proof of Corollary~\ref{cor:self-adjoint-case}]
  To be able to apply Theorem~\ref{thm:criterion-semigroups} we need $e^{t_1A}H\subseteq H_u$ for some $t_1>0$. This is either true by assumption or can be deduced if $D(A^\infty)\subseteq H_u$: since the semigroup is self-adjoint, it is analytic and hence, $e^{tA}H \subseteq D(A^\infty) \subseteq H_u$ for each $t > 0$. In either case the vector $u$ also defines a strictly positive functional on $H$. As noted above, we have $A' = A^*$ and thus $A' = A$ and so Assumption~(a) of Theorem~\ref{thm:criterion-semigroups} is satisfied with $t_1=t_2$. Therefore, the implication ``(iii) $\Rightarrow$ (i)'' follows from Theorem~\ref{thm:criterion-semigroups} since (iii) corresponds to Assumption~(b). The implication ``(i) $\Rightarrow$ (ii)'' is trivial, and the implication ``(ii) $\Rightarrow$ (iii)'' is a result from the theory of individually eventually positive semigroups, see \cite[Theorem~5.2 and Corollary~3.3]{Daners2016a} or \cite[Theorem~5.1]{Daners2017}.
\end{proof}

\section{Applications}
\label{section:applications}
In this section we demonstrate how our results can be applied to several concrete differential equations.

\paragraph*{The Dirichlet-to-Neumann operator}
Let $\Omega \subseteq \bbR^d$ by a bounded and smooth domain, say with
$C^\infty$-boundary. Let us consider the Dirichlet-to-Neumann operator
on $L^2(\partial \Omega)$ defined as follows. Fix a real number $\lambda$ which is not contained in the spectrum of the Dirichlet--Laplace operator on $L^2(\Omega)$. For $g \in L^2(\partial \Omega)$ we solve, if possible, the problem
\begin{displaymath}
  \begin{aligned}
    \Delta u &= \lambda u&& \text{on $\Omega$,} \\
    u &= g \quad && \text{on $\partial \Omega$}
  \end{aligned}
\end{displaymath}
for a function $u \in H^1(\Omega;\bbC)$; then we compute, again if
possible, the normal derivative $\frac{\partial}{\partial \nu}$ of that
solution $u$. The
mapping $g \mapsto \frac{\partial u}{\partial \nu}$ is the \emph{Dirichlet-to-Neumann operator} $D_\lambda$ on $L^2(\partial \Omega)$. It is a self-adjoint linear operator on $L^2(\partial \Omega)$ whose domain we denote by $D(D_\lambda)$. For a precise definition of the Dirichlet-to-Neumann operator we refer for instance to \cite{Arendt2014} or \cite{Arendt2012}. The spectral bound $\spb(-D_\lambda)$ of $-D_\lambda$ can be shown to be finite so that $-D_\lambda$ generates a self-adjoint $C_0$-semigroup on $L^2(\partial\Omega)$.

If $\lambda$ is sufficiently large, then the semigroup generated by $-D_\lambda$ is always positive. However, if $\Omega$ is the unit disk in $\bbR^2$, there exist choices of the parameter $\lambda$ for which the semigroup $(e^{-D_\lambda})_{t \in \bbR}$ is eventually positive, but not positive. This was proved by the first author in \cite{Daners2014} by using Fourier series. The abstract theory of eventually positive semigroups was used to analyse the Dirichlet-to-Neumann semigroup in \cite[Section~6.2]{Daners2016} (on the space $C(\partial \Omega)$ of continuous functions, where $\Omega$ is the two-dimensional unit disk) and in \cite[Proposition~6.8]{Daners2016a} (on $L^2(\partial \Omega)$, where $\Omega$ is a smooth and bounded domain in $\bbR^2$). In both cases, the abstract theory only yielded conditions for individual eventual positivity of $(e^{-tD_\lambda})_{t \ge 0}$. Corollary~\ref{cor:self-adjoint-case} now gives an immediate criterion for uniform eventual positivity and answers some questions posed in \cite[Section~5]{Daners2014}.

\begin{theorem}
  Let $\Omega \subseteq \bbR^d$ be a bounded domain with $C^\infty$-boundary and let $\lambda \in \bbR$ be contained in the resolvent set of the Dirichlet Laplacian on $L^2(\Omega)$. Then the following assertions are equivalent:
  \begin{enumerate}[\upshape (i)]
  \item The semigroup $(e^{-tD_\lambda})_{t \ge 0}$ is uniformly eventually strongly positive with respect to $\one$.
  \item The semigroup $(e^{-tD_\lambda})_{t \ge 0}$ is individually eventually strongly positive with respect to $\one$.
  \item The largest eigenvalue of $-D_\lambda$ is geometrically simple and admits an eigenfunction $v$ which fulfils $v \gg \one$.
  \end{enumerate}
\end{theorem}
\begin{proof}
  We note that by \cite[Theorem~2.3]{Elst2017} the operator $D_\lambda$ is self-adjoint on $L^2(\partial\Omega)$ and that $e^{-t_1D_\lambda}L^2(\partial \Omega)\subseteq L^\infty(\partial \Omega)=\bigl(L^2(\partial \Omega)\bigr)_{\one}$ for all $t>0$. Hence the assumptions of Corollary~\ref{cor:self-adjoint-case} are satisfied, completing the proof of the theorem.
\end{proof}

\paragraph*{The Laplace operator with non-local boundary conditions}
Fix a matrix $B \in \bbR^{2 \times 2}$ and a compact real interval $[a,b] \subseteq \bbR$. Let us consider the evolution equation
\begin{displaymath}
  \begin{aligned}
    \frac{\partial u}{\partial t} & = \Delta u
    &&\text{in $(a,b)\times (0,\infty)$} \\
    u(\cdot,0)& = u_0
    &&\text{on $(a,b)$,} \\
    \frac{\partial}{\partial \nu} u & = - Bu|_{\{a,b\}}
  \end{aligned}
\end{displaymath}
on $L^2((a,b))$, where $\frac{\partial}{\partial \nu} u = (-u'(a), u'(b))$ denotes the outer normal derivative. More precisely, this equation can be written as $\frac{\dx}{\dx t} u = \Delta_B u$, where $-\Delta_B\colon L^2([a,b]) \supseteq D(\Delta_B) \to L^2([a,b])$ is the operator associated to the the form
\begin{displaymath}
  a(u,v)=\int_a^b u'\overline{v'}\,\dx x + \langle Bu|_{\{a,b\}}, v|_{\{a,b\}} \rangle
\end{displaymath}
on the Sobolev space $H^1(a,b)$. It is not difficult to see that we actually have
\begin{displaymath}
  D(\Delta_B) = \{u \in H^2(a,b)\colon\frac{\partial}{\partial \nu} u = - Bu|_{\{a,b\}}\}.
\end{displaymath}
The operator $\Delta_B$ is real and generates an analytic $C_0$-semigroup on $L^2([a,b])$. In \cite[Section~6]{Daners2016a} it was shown that this semigroup in individually eventually strongly positive with respect to the constant function $\one$ with value one for certain choices of $B$. Using the theory developed in the present paper, we can now prove that the eventual positivity is even uniform.
\begin{theorem}
  Let
  \begin{math}
    B =
    \begin{pmatrix}
      1 & 1 \\
      1 & 1
    \end{pmatrix}.
  \end{math}
  Then the semigroup $(e^{t\Delta_H})_{t \ge 0}$ is not positive, but uniformly eventually strongly positive with respect to $\one$.
\end{theorem}
\begin{proof}
  It was shown in \cite[Theorem~6.11]{Daners2016a} that the semigroup is not positive, but individually eventually positive with respect to $\one$. Since the matrix $B$ is self-adjoint, so is the operator $\Delta_B$. Moreover, we have $D(\Delta_B) \subseteq H^1(a,b) \subseteq L^\infty([a,b]) = \big(L^2([a,b])\big)_{\one}$. Hence, the assumptions of Corollary~\ref{cor:self-adjoint-case} are fulfilled, so we conclude from the corollary that our semigroup is even uniformly eventually strongly positive with respect to $\one$.
\end{proof}

Let us now consider a case where $B$ is not symmetric and hence the semigroup generator $\Delta_H$ is not self-adjoint. The operator $\Delta_B$ for the following choice of $B$ occurs in \cite{Guidotti2000} where a one-dimensional model for a thermostat is studied.
\begin{theorem}
	Let $B =
	\begin{pmatrix}
		0 & \beta \\
		0 & 0
	\end{pmatrix}
	$
	for a real number $\beta \in (0,1/2)$. Then the semigroup $(e^{t\Delta_H})_{t \ge 0}$ is not positive, but uniformly eventually strongly positive with respect to $\one$.
\end{theorem}
\begin{proof}
	According to~\cite[Theorem~6.10]{Daners2016a} the semigroup is not positive, but individually eventually strongly positive with respect to $\one$. Let us show that the assumptions of Corollary~\ref{cor:ind-vs-unif} are fulfilled. An explicit computation in \cite[Theorem~6.1]{Guidotti2000} or \cite[Section~3]{Guidotti1997} shows that the spectral bound of $\Delta_B$ is an eigenvalue with an eigenfunction $v \gg_{\one} 0$. Since $\Delta_B$ is real, we have $\Delta_B^* = \Delta_B'$ as explained before Corollary~\ref{cor:self-adjoint-case}. Moreover, considering the adjoint operator $\Delta_B^*$ is equivalent to considering the transposed matrix of $B$, and this in turn simply means that both end points of the interval $[a,b]$ are switched. Hence, the adjoint $\Delta_B^* = \Delta_B'$ also has an eigenvector $\tilde v \gg_{\one} 0$ for the eigenvalue $\spb(\Delta_B)$. Furthermore, we have $D(\Delta_B) \subseteq H^1(a,b) \subseteq L^\infty([a,b]) = \big( L^2([a,b]) \big)_{\one}$, and the same is true for the domain of $\Delta_B^* = \Delta_B'$, again, since $\Delta_B^*$ has the same structure as $\Delta_B$, but with switched end points of the interval. Hence, all assumptions of Corollary~\ref{cor:ind-vs-unif} are fulfilled, so the assertion follows.
\end{proof}

\paragraph*{The bi-Laplace operator with Dirichlet boundary conditions}
Let $\Omega \subseteq \bbR^d$ be a bounded domain with sufficiently smooth boundary, say $C^\infty$, and fix $p \in (1,\infty)$. We consider the evolution equation
\begin{displaymath}
  \begin{aligned}
    \dot u(t) &= -\Delta^2u(t)
    &&\text{in $\Omega\times(0,\infty)$,} \\
    u(0) &= u_0
    &&\text{in $\Omega$,} \\
    \gamma(u) &= \frac{\partial u}{\partial \nu} = 0
    &&\text{on $\partial\Omega\times(0,\infty)$,}
  \end{aligned}
\end{displaymath}
on $L^p := L^p(\Omega;\bbC)$, where $u_0 \in L^p$ is an initial function, where $\gamma\colon L^p(\Omega) \to L^p(\partial \Omega)$ denotes the trace operator and where $\frac{\partial}{\partial \nu}$ denotes the normal derivative. This evolution equation can be rewritten in an abstract form as $\dot u(t) = A_p  u(t)$, $u(0) = u_0$, if we encode the boundary conditions in the domain $D(A_p)$ of the operator $A_p$. More precisely, we define
\begin{align}
	\label{eq:bi-laplace-dirichlet}
	D(A_p) = W^{4,p}(\Omega;\bbC) \cap W^{2,p}_0(\Omega;\bbC), \qquad A_p u = -\Delta^2 u.
\end{align}
Positivity and eventual properties of this operator are a quite prevalent topic in the literature. For instance, the question whether the inverse of $-A_p$ is positive has its origins over a century ago (see e.g.\ \cite{Grunau1998} for a bit more historic information) and it turned out the (eventual) positivity properties of $A_p$ are closely related to the shape of the domain $\Omega$. For arbitrary domains one can not even expect that the eigenfunction for the first eigenvalue of $A_p$ is positive and hence, the results in \cite[Sections~3--5]{Daners2016a} and \cite[Sections~3--5]{Daners2017} show that we cannot expect any eventual positivity properties of $A_p$. If, however, the domain $\Omega$ is in a sense close to the unit ball, then positivity properties for the first eigenfunction of $A_p$ were, for instance, obtained in \cite[Section~5]{Grunau1998}. In \cite[Section~6]{Daners2016a} it was shown that in such a case the semigroup generated by $A_p$ is individually eventually strongly positive. However, \emph{uniform} eventual strong positivity is as today only known for the case $p = 2$ where it can be derived from the already known Corollary~\ref{cor:self-adjoint-case} (see \cite[Theorem~10.2.1 and Theorem~11.4.3]{GlueckDISS}). Using the results of the present paper we can now prove uniform eventual strong positivity on $L^p$ for every $p \in (1,\infty)$.

\begin{theorem}
  \label{thm:bi-laplace-dirichlet}
  Let $p \in (1,\infty)$ and let $\Omega \subseteq \bbR^d$ by a bounded domain with $C^\infty$-boundary which is sufficiently close to the unit ball in the sense of~\emph{\cite[Theorem~5.2]{Grunau1998}}. Let $A_p$ denote the bi-Laplace operator with Dirichlet boundary conditions on $L^p := L^p(\Omega;\bbC)$ which is given by~\eqref{eq:bi-laplace-dirichlet}.
	
  Then $A_p$ generates an analytic $C_0$-semigroup $(e^{tA_p})_{t \ge 0}$ on $L^p$ which is uniformly eventually strongly positive with respect to the function $u$ given by $u(\omega) := \dist(\omega,\partial \Omega)^2$ for all $\omega \in \Omega$.
\end{theorem}
\begin{proof}
	According to \cite[Proposition~6.6]{Daners2016a} or \cite[Theorem~11.4.2]{GlueckDISS}, the semigroup is individually eventually strongly positive with respect to $u$. As explained in the proof of \cite[Proposition~6.6]{Daners2016a}, we have $D(A_p^n) \subseteq (L^p)_u$ for all sufficiently large $n \in \bbN$; hence, $e^{tA_p} L^p \subseteq (L^p)_u$ for all $t > 0$ since our semigroup is analytic. Moreover, the eigenspace $\ker(s(A_p) - A_p)$ is spanned by a function $v \gg_u 0$ according to \cite[Theorem~5.2]{Grunau1998}. Finally, it is not difficult to see that the dual operator $A_p'$ of $A_p$ coincides with the operator $A_q$ for $1/p + 1/q = 1$. To see that use for instance \cite[Proposition~8.3.1]{GlueckDISS}. Thus, $A_p'$ has similar properties as $A_p$ and hence, the assumptions of Corollary~\ref{cor:ind-vs-unif} are fulfilled and the assertion of the theorem follows.
\end{proof}

\begin{remark}
  Let the notation be as in Theorem~\ref{thm:bi-laplace-dirichlet}. For $p > 2$ the theorem also follows from the $L^2$-case (which has already been discussed in \cite[Theorem~11.4.3]{GlueckDISS}) since $L^p \subseteq L^2$.
\end{remark}

\paragraph*{A delay differential equation}

In \cite[Section~6.5]{Daners2016} and \cite[Section~11.6]{GlueckDISS} it was demonstrated that individual eventual positivity can occur for certain delay differential equations. Here, we are going to show by means of a concrete example that delay differential equations can also exhibit uniformly eventually positive behaviour. Fix a real number $c > 0$ and consider the evolution equation
\begin{equation}
  \label{eq:delay}
  \dot y(t)= c\left(\int_{t-2}^{t-1} y(s)\,\dx s
    -\int_{t-1}^t y(s)\,\dx s + y(t-2) - y(t)\right)
  \qquad \text{for } t \ge 0,
\end{equation}
where $y\colon [-2,\infty) \to \bbC$ is the wanted function. In order for~\eqref{eq:delay} to have a unique solution in any sense, we need to prescribe the values of $y$ on the interval $[-2,0]$. Delay equation of this type can, for instance, be treated by semigroup methods, and there are several possible choices for the underlying Banach space, the so-called \emph{state space}. A very common approach, which can for instance be found in \cite[Section~VI.6]{Engel2000}, is to use $C([-2,0])$ as state space. However, for the purpose of proving uniform eventual positivity, this space is not a good choice since the dual space of $C([-2,0])$ is extremely large and thus, we will have difficulties to check the second domination condition in Theorem~\ref{thm:criterion-semigroups}(a).

An alternative to formulate~\eqref{eq:delay} by means of a semigroup is to choose an $L^p$-space for some $p \in [1,\infty)$ as state space. This approach is for instance explained in the monograph \cite{Batkai2005}, and we will use it here for $p = 1$. We choose
\begin{displaymath}
  E := \bbC \times L^1([-2,0];\bbC)
\end{displaymath}
as our state space and we define a closed, densely defined linear operator $A\colon E \supseteq D(A) \to E$ by
\begin{displaymath}
  \begin{aligned}
    D(A)&:=\{(x, f)\in\bbC\times W^{1,1}((-2,0);\bbC)\colon f(0) = x\},\\
    A(x,f)&:=(\langle \Phi, f\rangle, f'),
  \end{aligned}
\end{displaymath}
where $\Phi$ is the continuous linear functional on $W^{1,1}((-2,0);\bbC)$ given by
\begin{displaymath}
  \langle \Phi, f \rangle
  := c \left(\int_{-2}^{-1} f(s)\,\dx s-\int_{-1}^0 f(s)\,\dx s+f(-2)-f(0)\right)
\end{displaymath}
for each $f \in W^{1,1}((-2,0);\bbC)$. According to \cite[Theorem~3.23]{Batkai2005}, $A$ generates a $C_0$-semigroup $(e^{tA})_{t \ge 0}$ on $E$. As explained in \cite[Section~3.1]{Batkai2005}, the solutions of the delay equation~\eqref{eq:delay} can be described by this semigroup. Noter that the monograph \cite{Batkai2005} only deals with delays on the time interval $[-1,0]$, but considering a different time interval does not change any of the results. Note that the operator $A$ is real, and hence so is the semigroup generated by $A$.

\begin{theorem}
  Let $E$ and $A\colon E \supseteq D(A) \to E$ be as described above and set $u := (1,\one_{[-2,0]}) \in E$. If $c\le\pi/16$, then the semigroup $(e^{tA})_{t \ge 0}$ is uniformly eventually strongly positive with respect to $u$.
\end{theorem}
\begin{proof}
  We first note that $u$ is an eigenvector of $A$ for the eigenvalue $0$; this follows immediately from the definition of $A$. Moreover, it is easy to see that $\ker A$ is one-dimensional and spanned by $u$. Let us now show that $e^{2A}E \subseteq E_u = \bbC \times L^\infty([-2,0];\bbC)$.
	
  By $\pi_1\colon E \to \bbC$ we denote the projection onto the first component, i.e.\ $\pi(x,f) = x$ for all $(x,f) \in E$. We define a linear operator $S\colon E \to E$ in the following way: For all $(x,f) \in E$ and all $t \ge 0$, let $w_{(x,f)}(t) := \pi_1 e^{tA}(x,f)$; note that $w_{(x,f)}$ is a continuous mapping from $[0,\infty)$ to $\bbC$ since the semigroup $(e^{tA})_{t \ge 0}$ is strongly continuous. We set
  \begin{displaymath}
    S(x,f) = \bigl(w_{(x,f)}(2), w_{(x,f)}(2 + \phdot) \bigr)
    \quad \text{for all } (x,f) \in E.
  \end{displaymath}
  Clearly, $S$ is a linear mapping from $E$ to $E$, and in fact it maps $E$ to $\bbC \times C([-2,0];\bbC)$. Moreover, one readily checks that $S$ is continuous.
	
  We claim that $S = e^{2A}$. For $(x,f) \in D(A)$, it follows from the last assertion in \cite[Proposition~3.9]{Batkai2005} that $S(x,f) = e^{2A}(x,f)$. Since $D(A)$ is dense in $E$ we conclude that $S = e^{2A}$ as claimed. Hence, $e^{2A}E = SE \subseteq \bbC \times C([-2,0];\bbC) \subseteq E_u$.

  Now we consider the functional $\varphi \in E'$ which is given by 
  \begin{displaymath}
    \langle\varphi, (x,f) \rangle = x + c\left( \int_{-2}^{-1} (2+s) f(s)\, \dx s + \int_{-1}^0 -s f(s) \, \dx s + \int_{-2}^0 f(s) \, \dx s \right)
  \end{displaymath}
  for all $(x,f) \in E$. A straightforward computation shows that $\varphi \in \ker(A')$. We may identify $E'$ with $\bbC \times L^\infty([-2,0];\bbC)$, and under this identification $\varphi$ can be represented as the vector $(1, v)$, where $v(s) = 3+s$ for $s \in [-2,-1]$ and $v(s) = 1-s$ for $s \in (-1,0]$. Hence, $\varphi = (1,v) \ge (1, \one)$, which shows that $\varphi$ is a strictly positive functional and that even $(E')_\varphi = E'$. In particular, we have $e^{tA'} \subseteq (E')_\varphi$ for all $t \ge 0$.
	
  In order to apply Theorem~\ref{thm:criterion-semigroups}, it only remains to show that $0$ is the spectral bound of $A$ and that it is even a dominant spectral value of $A$. It follows from Corollary~\ref{cor:domination-and-compactness} that our semigroup $(e^{tA})_{t \ge 0}$ is eventually compact, so all spectral values of $A$ are in fact eigenvalues. It is easy to see that a complex number $\lambda$ is an eigenvalue of $A$ if and only if $\lambda  = \langle \Phi, e^{\lambda\phdot}\rangle$ (see also \cite[Section~3.2]{Batkai2005} for more sophisticated spectral results).
	
  Now, let $\lambda = \alpha + i\beta$ be an eigenvalue of $A$ where $\alpha,\beta \in \bbR$. Assume that $\alpha \ge 0$. We have to show that $\lambda = 0$. As the operator $A$ is real, all of its eigenvalues occur in complex conjugate pairs, so we may also assume that $\beta \ge 0$. By splitting the equation $\lambda  = \langle \Phi, e^{\lambda\phdot}\rangle$ into its real and imaginary part we obtain
  \begin{align}
    \label{eq:alpha} \alpha
    & = \langle \Phi, e^{\alpha \phdot} \cos(\beta \phdot)\rangle, \\
    \label{eq:beta}  \beta
    &  = \langle \Phi, e^{\alpha \phdot} \sin(\beta \phdot)\rangle.
  \end{align}
  Since the modulus of the function $e^{\alpha \phdot} \sin(\beta \phdot)$ on the interval $[-2,0]$ is bounded by $1$, equation~\eqref{eq:beta} and the definition of $\Phi$ yield $\beta = |\beta| \le 4c$. Since $c \le \pi/16$ by assumption, we conclude that $2\beta \le \frac{\pi}{2}$. Hence, the function $\cos(\beta \phdot)$ is non-negative and increasing on the interval $[-2,0]$. Now assume for a contradiction that $\lambda \not= 0$. Then $\alpha > 0$ or $\beta > 0$, so the non-negative function $e^{\alpha \phdot} \cos(\beta\phdot)$ on $[-2,0]$ is even strictly increasing. A short glance at the definition of $\Phi$ shows that this implies $\langle \Phi, e^{\alpha \phdot} \cos(\beta \phdot) \rangle < 0$, which contradicts~\eqref{eq:alpha}.
	
  Hence, $\lambda = 0$ and we conclude that $0$ is indeed the spectral bound and a dominant spectral value of $A$. Thus, all assumptions of Theorem~\ref{thm:criterion-semigroups} are fulfilled and the theorem yields the assertion.
\end{proof}

\paragraph*{Acknowledgements.}
This paper was initiated during a very pleasant stay of the second author at the University of Sydney. Part of the work was done while the second author was financially supported by a scholarship within the scope of the Landesgraduiertenf\"orderung Baden--W\"urttemberg (grant number 130l LGFG-E).

\pdfbookmark[1]{\refname}{biblio}%
\bibliographystyle{doi}%
\bibliography{literature_evpos}

\end{document}